\documentclass[10pt]{amsart}
\allowdisplaybreaks

\usepackage{fourier}
\usepackage[utf8]{inputenc}
\usepackage[T1]{fontenc}

\usepackage{amsmath,amssymb,amsthm}
\usepackage{mathtools}
\mathtoolsset{mathic=true}
\usepackage{enumerate}
\usepackage{srcltx}
\usepackage{cite}
\usepackage{xspace}

\usepackage{tikz}
\usetikzlibrary{decorations.pathreplacing}
\usetikzlibrary{patterns}
\usetikzlibrary{shapes.misc}
\tikzset{>=latex}
\tikzset{cross/.style={cross out, draw=black, minimum size=2*(#1-\pgflinewidth), inner sep=0pt, outer sep=0pt},
  cross/.default={1pt}}

\usepackage{bookmark}
\usepackage{hyperref}

\usepackage{subcaption}
\usepackage{graphicx}

\usepackage{xcolor}

\newtheorem{theorem}{Theorem}[section]

\newtheorem*{theorem-nn}{Theorem}
\newtheorem{corollary}[theorem]{Corollary}
\newtheorem{lemma}[theorem]{Lemma}
\newtheorem{proposition}[theorem]{Proposition}

\newtheorem*{question-nn}{Question}

\newtheorem*{conjecture-nn}{Conjecture}
\theoremstyle{definition}
\newtheorem{definition}[theorem]{Definition}
\newtheorem*{definition-nn}{Definition}
\theoremstyle{remark}
\newtheorem*{claim*}{Claim}

\newtheorem{remark}[theorem]{Remark}
\newtheorem{example}[theorem]{Example}
\newtheorem*{example-nn}{Example}

\newcommand{\acts}{\curvearrowright}
\newcommand{\dist}{\mathrm{dist}}
\newcommand{\diam}{\mathrm{diam}}

\newcommand{\norm}[1]{\left\lVert #1\right\rVert}

\DeclareMathOperator*{\inter}{\mathrm{int}}

\begin{document}
\title{Katok's special representation theorem \\ for multidimensional Borel flows}

\author{Konstantin Slutsky}
\address{Department of Mathematics \\
  Iowa State University \\
  411 Morrill Road \\
  Ames, IA 50011}
\email{kslutsky@gmail.com}
\thanks{The author was partially supported by NSF grant DMS-2153981}

\begin{abstract}
  Katok's special representation theorem states that any free ergodic measure-preserving
  \(\mathbb{R}^{d}\)-flow can be realized as a special flow over a \(\mathbb{Z}^{d}\)-action.  It
  provides a multidimensional generalization of the ``flow under a function'' construction.  We
  prove the analog of Katok's theorem in the framework of Borel dynamics and show that, likewise,
  all free Borel \(\mathbb{R}^{d}\)-flows emerge from \(\mathbb{Z}^{d}\)-actions through the special
  flow construction using bi-Lipschitz cocycles.
\end{abstract}

\maketitle
%%%%%%%%%%

\section{Introduction}
\label{sec:introduction}

\subsection{Overview}
\label{sec:overview}

Theorems of Ambrose and
Kakutani~\cite{ambroseRepresentationErgodicFlows1941,ambroseStructureContinuityMeasurable1942}
established a connection between measure-preserving \(\mathbb{Z}\)-actions and \(\mathbb{R}\)-flows
by showing that any flow admits a cross-section and can be represented as a ``flow under a
function''. Their construction provides a foundation for the theory of Kakutani equivalence (also
called monotone
equivalence)~\cite{feldmanNewAutomorphismsProblem1976,katokMonotoneEquivalenceErgodic1977} on the
one hand and study of the possible ceiling functions in the ``flow under a function''
representation~\cite{rudolphTwovaluedStepCoding1976,krengelRudolphRepresentationAperiodic1976} on
the other.

The intuitive geometric picture of a ``flow under a function'' does not generalize to
\(\mathbb{R}^{d}\)-flows for \(d \ge 2\).  However,
Katok~\cite{katokSpecialRepresentationTheorem1977} re-interpreted it in a way that can readily be
adapted to the multidimensional setup, calling flows appearing in this construction \emph{special
  flows}. Despite their name, they aren't so special, since, as showed in the same paper, every free
ergodic measure-preserving \(\mathbb{R}^{d}\)-flow is metrically isomorphic to a special flow. Just
like the works of Ambrose and Kakutani, it opened gates for the study of multidimensional concepts
of Kakutani equivalence~\cite{deljuncoKakutaniEquivalenceErgodic1984} and stimulated research on
tilings of flows\cite{rudolphRectangularTilingsMathbbR1988,kraRudolphTwoStep2015}.

Borel dynamics as a separate field goes back to the work of Weiss~\cite{weissMeasurableDynamics1984}
and has blossomed into a versatile branch of dynamical systems.  The phase space here is a standard
Borel space \((X, \mathcal{B})\), i.e., a set \(X\) with a \(\sigma\)-algebra \(\mathcal{B}\) of
Borel sets for some Polish topology on \(X\).  Some of the key ergodic theoretical results have
their counterparts in Borel dynamics, while others do not generalize. For example, Borel version of
the Ambrose--Kakutani Theorem on the existence of cross-sections in \(\mathbb{R}\)-flows was proved
by Wagh in~\cite{waghDescriptiveVersionAmbrose1988} showing that, just like in ergodic theory, all
free Borel \(\mathbb{R}\)-flows emerge as ``flows under a function'' over Borel
\(\mathbb{Z}\)-actions.  Likewise, Rudolph's two-valued
theorem~\cite{rudolphTwovaluedStepCoding1976} generalizes to the Borel
framework~\cite{slutskyRegularCrossSections2019}. The theory of Kakutani equivalence, on the other
hand, exhibits a different phenomenon.  While being a highly non-trivial equivalence relation among
measure-preserving
flows~\cite{ornsteinEquivalenceMeasurePreserving1982,gerberAnticlassificationResultsKakutani2021},
descriptive set theoretical version of Kakutani equivalence collapses
entirely~\cite{millerDescriptiveKakutaniEquivalence2010}.

Considerable work has been done to understand the Borel dynamics of \(\mathbb{R}\)-flows, but
relatively few things are known about multidimensional actions.  This paper makes a contribution in
this direction by showing that the analog of Katok's special representation theorem does hold for
free Borel \(\mathbb{R}^{d}\)-flows.

\subsection{Structure of the paper}
\label{sec:structure-paper}

Constructions of orbit equivalent \(\mathbb{R}^{d}\)-actions often rely on (essential)
hyperfiniteness and use covers of orbits of the flow by coherent and exhaustive regions.  This is
the case for the aforementioned paper of Katok~\cite{katokSpecialRepresentationTheorem1977}, and
related approaches have been used in the descriptive set theoretical setup as well
(e.g.,~\cite{slutskySmoothOrbitEquivalence2021}). Particular assumptions on such coherent regions,
however, depend on the specific application.  Section~\ref{sec:seq-partial-actions} distills a
general language of partial actions, in which many of the aforementioned constructions can be
formulated.  As an application we show that the orbit equivalence relation generated by a free
\(\mathbb{R}\)-flow can also be generated by a free action of \emph{any} non-discrete and
non-compact Polish group (see Theorem \ref{thm:flow-generated-by-any-Polish-group}).  This is in a
striking contrast with the actions of discrete groups, where a probability measure-preserving free
\(\mathbb{Z}\)-action can be generated only by a free action of an amenable group.

Section~\ref{sec:lipsch-maps} does the technical work of constructing Lipschitz maps that are needed
for Theorem~\ref{thm:grid-flow}, which shows, roughly speaking, that up to an arbitrarily small
bi-Lipschitz perturbation, any free \(\mathbb{R}^{d}\)-flow admits an integer grid---a Borel
cross-section invariant under the \(\mathbb{Z}^{d}\)-action.

Finally, Section~\ref{sec:spec-repr-theor} discusses the descriptive set theoretical version of
Katok's special flow construction and shows in Theorem~\ref{thm:katoks-special-representation-flows}
that, indeed, any free \(\mathbb{R}^{d}\)-flow can be represented as a special flow generated by a
bi-Lipschitz cocycle with Lipschitz constants arbitrarily close to \(1\).  This provides a Borel
version of Katok's special representation theorem.

%%% Local Variables:
%%% mode: latex
%%% TeX-master: "../Katok-representation-theorem-for-Borel-flows"
%%% End:

\section{Sequences of partial actions}
\label{sec:seq-partial-actions}

We begin by discussing the framework of partial actions suitable for constructing orbit equivalent
actions. Throughout this section, \(X\) denotes a standard Borel space.

\subsection{Partial actions}
\label{sec:part-acti-tess}
Let \(G\) be a standard Borel group, that is a group with a structure of a standard Borel space that
makes group operations Borel. A \textbf{partial \(G\)-action}\footnote{More precisely, we should
  call such \((E, \phi)\) a partial \emph{free} action.  Since we are mainly concerned with free
  actions in what follows, we choose to omit the adjective ``free'' in the definition.} is a pair
\((E, \phi)\), where \(E\) is a Borel equivalence relation on \(X\) and \(\phi : X \to G\) is a
Borel map that is injective on each \(E\)-class: \(\phi(x) \ne \phi(y)\) whenever \(x E y\).  The
map \(\phi\) itself may occasionally be refer to as a partial action when the equivalence relation
is clear from the context.

The motivation for the name comes from the following observation.  Consider the set
\[A_{\phi} = \bigl\{(g,x, y) \in G \times X \times X : x E y \textrm{ and } g\phi(x) =
  \phi(y)\bigr\}.\]
Injectivity of \(\phi\) on \(E\)-classes ensures that for each \(x \in X\) and \(g \in G\) there is
at most one \(y \in X\) such that \((g,x,y) \in A_{\phi}\).  When such a \(y\) exists, we say that
the action of \(g\) on \(x\) is defined and set \(gx = y\).  Clearly, \((e,x, x) \in A_{\phi}\) for
all \(x \in X\), thus \(ex=x\); also \(g_{2}(g_{1}x) = (g_{2}g_{1})x\) whenever all the terms are
defined.  The set \(A_{\phi}\) is a graph of a total action \(G \acts X\) if and only if for each
\(x \in X\) and \(g \in G\) there does exist some \(y \in X\) such that \((g,x,y) \in A_{\phi}\); in
this case the orbit equivalence relation generated by the action coincides with \(E\).

\begin{example}
  \label{exmpl:restriction-of-total-action}
  An easy way of getting a partial action is by restricting a total one.  Suppose we have a free
  Borel action \(G \acts X\) with the corresponding orbit equivalence relation \(E_{G}\) and suppose
  that a Borel equivalence sub-relation \(E \subseteq E_{G}\) admits a Borel selector---a Borel
  \(E\)-invariant map \(\pi : X \to X\) such that \(x E \pi(x)\) for all \(x \in X\). If
  \(\phi : X \to G\) is the map specified uniquely by the condition \(\phi(x) \pi(x) = x\), then
  \((E, \phi)\) is a partial \(G\)-action.
\end{example}

\medskip

Sub-relations \(E\) as in Example~\ref{exmpl:restriction-of-total-action} are often associated with
cross-sections of actions of locally compact second countable (lcsc) groups.

\subsection{Tessellations of lcsc group actions}
\label{sec:tess-locally-comp}
Consider a free Borel action \(G \acts X\) of a locally compact second countable group.  A
\textbf{cross-section} of the action is a Borel set \(\mathcal{C} \subseteq X\) that intersects
every orbit in a countable non-empty set.  A cross-section \(\mathcal{C} \subseteq X\) is
\begin{itemize}
\item \textbf{discrete} if \((Kx) \cap \mathcal{C}\) is finite for every \(x \in X\) and compact \(K
  \subseteq G\);
\item \textbf{\(U\)-lacunary}, where \(U \subseteq G\) is a neighborhood of the identity, if
  \(Uc \cap \mathcal{C} = \{c\}\) for all \(c \in \mathcal{C}\);
\item \textbf{lacunary} if it is \(U\)-lacunary for some neighborhood of the identity \(U\);
\item \textbf{cocompact} if \(K\mathcal{C} = X\) for some compact \(K \subseteq G\).
\end{itemize}
Let \(\mathcal{C}\) be a lacunary cross-section for \(G \acts X\), which exists
by~\cite[Corollary~1.2]{kechrisCountableSectionsLocally1992}.  Any lcsc group \(G\) admits a
compatible left-invariant proper metric~\cite{strubleMetricsLocalLycompact1974}, and any
left-invariant metric \(d\) can be transferred to orbits due to freeness of the action via
\(\dist(x,y) = d(g,e)\) for the unique \(g \in G\) such that \(gx = y\).  One can now define the
so-called Voronoi tessellation of orbits by associating with each \(x \in X\) the closest point
\(\pi_{\mathcal{C}}(x) \in \mathcal{C}\) of the cross-section \(\mathcal{C}\) as determined by
\(\dist\).  Properness of the metric ensures that, for a ball \(B_{R} \subseteq G\) of radius \(R\),
\(B_{R} = \{g \in G : d(g,e) \le R\}\), and any \(x \in X\), the set \(\mathcal{C} \cap B_{R}x\) is
finite.  Indeed, there can be at most \(\lambda(B_{R +r})/ \lambda(B_{r})\) points in the
intersection, where \(\lambda\) is a Haar measure on the group and \(r > 0\) is so small that
\(B_{r}c \cap B_{r} c' = \varnothing\) whenever \(c, c' \in \mathcal{C}\) are distinct.

Small care needs to be taken to address the possibility of having several closest
points.  For example, one may pick a Borel linear order on \(\mathcal{C}\) and associated each \(x\)
with the \emph{smallest} closest point in the cross-section
(see~\cite[Section~4]{slutskyLebesgueOrbitEquivalence2017}
or~\cite[Section~B.2]{maitreMathrmLFullGroups2021} for the specifics). This way we get a Borel
equivalence relation \(E_{\mathcal{C}} \subseteq E_{G}\) whose equivalence classes are the cells of
the Voronoi tessellation: \(x E_{\mathcal{C}} y\) if and only if
\(\pi_{\mathcal{C}}(x) = \pi_{\mathcal{C}}(y)\).

Assumed freeness of the action \(G \acts X\) allows for a natural identification of each Voronoi
cell with a subset of the acting group via the map
\(\pi_{\mathcal{C}}^{-1}(c) \ni x \mapsto \phi_{\mathcal{C}}(x) \in G\) such that
\(\phi_{\mathcal{C}}(x)c = x\), which is exactly what the corresponding partial action from
Example~\ref{exmpl:restriction-of-total-action} does.

\medskip

Our intention is to use partial actions to define total actions, and the example above may seem like
going ``in the wrong direction''.  The point, however, is that once we have a partial action
\(\phi : X \to G\), we can compose it with an arbitrary Borel injection \(f : G \to G\) to get a
different partial action \(f\circ \phi\).  This pattern is typical in the sense that new partial
actions are often constructed by modifying those obtained as restrictions of total actions.

\subsection{Convergent sequences of partial actions}
\label{sec:sequ-part-acti}

A total action can be defined whenever we have a sequence of partial actions that cohere in the
appropriate sense. Let \(G\) be a standard Borel group. A sequence \((E_{n}, \phi_{n})\),
\(n \in \mathbb{N}\), of partial \(G\)-actions on \(X\) is said to be \textbf{convergent} if it
satisfies the following properties:
\begin{itemize}
\item \textbf{monotonicity:} equivalence relations \(E_{n}\) form an increasing sequence, that is
  \(E_{n} \subseteq E_{n+1}\) for all \(n\);
\item \textbf{coherence:} for each \(n\) the map \(x \mapsto (\phi_{n}(x))^{-1}\phi_{n+1}(x)\) is
  \(E_{n}\)-invariant;
\item \textbf{exhaustiveness:} for all \(x \in X\) and all \(g \in G\) there exist \(n\) and
  \(y \in X\) such that \(xE_{n}y\) and \(g\phi_{n}(x) = \phi_{n}(y)\).
\end{itemize}
With such a sequence one can associate a free Borel (left) action \(G \acts X\), called the
\textbf{limit} of \((E_{n}, \phi_{n})_{n}\), whose graph is \(\bigcup_{n}A_{\phi_{n}}\).  Coherence
ensures that the partial action defined by \(\phi_{n+1}\) is an extension of the one given by
\(\phi_{n}\).  Indeed, if \(xE_{n}y\) are such that \(g\phi_{n}(x) = \phi_{n}(y)\), then also
\(x E_{n+1} y\) by monotonicity and, using coherence,
\begin{displaymath}
  g \phi_{n+1}(x) = g \phi_{n}(x)(\phi_{n}(x))^{-1}\phi_{n+1}(x) = \phi_{n}(y)
  (\phi_{n}(y))^{-1}\phi_{n+1}(y) = \phi_{n+1}(y),
\end{displaymath}
whence \(A_{\phi_{n}} \subseteq A_{\phi_{n+1}}\).  If \(C\) is an \(E_{n}\)-class,
and \(s = (\phi_{n}(x))^{-1}\phi_{n+1}(x)\) for some \(x \in C\), then \(\phi_{n+1}(C) = \phi_{n}(C)s\), so the image
\(\phi_{n}(C)\) gets shifted on the right inside \(\phi_{n+1}(C)\).  If we want to build a
right action of the group, then \(\phi_{n}(C)\) should be shifted on the left instead.

Finally, exhaustiveness guarantees that \(gx\) gets defined eventually: for all \(g \in G\) and
\(x \in X\) there are \(n\) and \(y \in X\) such that \((g,x,y) \in A_{\phi_{n}}\). It is
straightforward to check that \(\bigcup_{n}A_{\phi_{n}}\) is a graph of a total Borel
action \(G \acts X\).  Equally easy is to check that the action is free, and its orbits are precisely
the equivalence classes of \(\bigcup_{n}E_{n}\).

\medskip

This framework, general as it is, delegates most of the complexity to the construction of maps
\(\phi_{n}\). Let us illustrate these concepts on essentially hyperfinite actions of lcsc groups.

\subsection{Hyperfinite tessellations of lcsc group actions}
\label{sec:hyperf-tess-lcsc}
In the context of Section~\ref{sec:tess-locally-comp}, suppose that, furthermore, the restriction of
the orbit equivalence relation \(E_{G}\) onto the cross-section \(\mathcal{C}\) is hyperfinite,
i.e., there is an increasing sequence of finite Borel equivalence relations \(F_{n}\) on
\(\mathcal{C}\) such that \(\bigcup_{n}F_{n} = E_{G}|_{\mathcal{C}}\).  We can use this sequence to
define \(xE_{n}y\) whenever \(\pi_{\mathcal{C}}(x)F_{n}\pi_{\mathcal{C}}(y)\), which yields an
increasing sequence of Borel equivalence relations \(E_{n}\) such that \(E_{G} = \bigcup_{n}E_{n}\).

The equivalence relations \(F_{n}\) admit Borel transversals, i.e., there are Borel sets
\(\mathcal{C}_{n}\) that pick exactly one point from each \(F_{n}\)-class.  Just as in
Section~\ref{sec:tess-locally-comp}, we may define \(\phi_{n}(x) \) to be such an element
\(g \in G\) that \(gc = x\) for the unique \(c \in \mathcal{C}_{n}\) satisfying \(x E_{n}c\).  This
gives a convergent sequence of partial \(G\)-actions \((E_{n}, \phi_{n})_{n}\) whose limit is the
original action \(G \acts X\).

\subsection{Partial actions revisited}
\label{sec:part-acti-part}

In practice, it is often more convenient to allow equivalence relations \(E_{n}\) to be defined on
proper subsets of \(X\).  Let \(X_{n} \subseteq X\), \(n \in \mathbb{N}\), be Borel subsets, and
suppose for each \(n\), \(E_{n}\) is a Borel equivalence relation on \(X_{n}\).  We say that the
sequence \((E_{n})_{n}\) is \textbf{monotone} if the following conditions are satisfied for all
\(m \le n\):
\begin{itemize}
\item \(E_{m}|_{X_{m} \cap X_{n}} \subseteq E_{n}|_{X_{m} \cap X_{n}}\);
\item if \(x \in X_{m}\cap X_{n}\) then the whole \(E_{m}\)-class of \(x\) is in \(X_{n}\).
\end{itemize}

Partial action maps \(\phi_{n} : X_{n} \to G\), where, as earlier, \(G\) is a standard Borel group,
need to satisfy the appropriate versions of coherence and exhaustiveness:
\begin{itemize}
\item \textbf{coherence}: \(X_{m} \cap X_{n} \ni x \mapsto (\phi_{m}(x))^{-1}\phi_{n}(x)\) is
\(E_{m}\)-invariant for each \(m < n\);
\item \textbf{exhaustiveness}: for each \(x \in X\) and \(g \in G\) there exist \(n\) and
  \(y \in X_{n}\) such that \(x \in X_{n}\), \(x E_{n}y\), and \(g\phi_{n}(x) = \phi_{n}(y) \).
\end{itemize}
A sequence of partial \(G\)-actions \((X_{n},E_{n},\phi_{n})_{n}\) will be called
\textbf{convergent} if it satisfies the above properties of monotonicity, coherence, and
exhaustiveness.  Note that the condition \(\bigcup_{n}X_{n} = X\) follows from exhaustiveness, so
sets \(X_{n}\) must cover all of \(X\).

Convergent sequences \((X_{n},E_{n},\phi_{n})_{n}\) define total actions, which can be most easily
seen by reducing this setup to the notationally simpler one given in
Section~\ref{sec:sequ-part-acti}. To this end, extend \(E_{n}\) to the equivalence relation
\(\hat{E}_{n}\) on all of \(X\) by
\[x\hat{E}_{n}y \iff \exists m \le n\ xE_{m}y \textrm { or } x = y;\]
and also extend \(\phi_{n}\) to \(\hat{\phi}_{n} : X \to G\) by setting
\(\hat{\phi}_{n}(x) = \phi_{m}(x)\) for the maximal \(m \le n\) such that \(x \in X_{m}\) or
\(\hat{\phi}_{n}(x) = e\) if no such \(m\) exists.  It is straightforward to check that
\((\hat{E}_{n}, \hat{\phi}_{n})_{n}\) is a convergent sequence of partial \(G\)-actions in the sense
of Section~\ref{sec:sequ-part-acti}.  By the \textbf{limit} of the sequence of partial actions
\((X_{n}, E_{n}, \phi_{n})_{n}\) we mean the limit of \((\hat{E}_{n}, \hat{\phi}_{n})_{n}\) as
defined earlier.

\begin{remark}
  \label{rem:layered-variant-partial-actions}
  A variant of this generalized formulation, which we encounter in
  Proposition~\ref{prop:chain-bers-flow} below, occurs when sets \(X_{n}\) are nested:
  \(X_{0} \subseteq X_{1} \subseteq X_{2} \subseteq \cdots\). Monotonicity of equivalence relations
  then simplifies to \(E_{0} \subseteq E_{1} \subseteq E_{2} \subseteq \cdots \) and coherence
  becomes equivalent to the \(E_{n}\)-invariant of maps
  \( X_{n} \ni x \mapsto (\phi_{n}(x))^{-1}\phi_{n+1}(x) \in G\).
\end{remark}

\medskip

As was mentioned above, it is easy to create new partial actions simply by composing a partial
action \(\phi : X \to G\) with some Borel bijection \(f : G \to G\) (or \(f : G \to H\) if we choose
to have values in a different group).  However, an arbitrary bijection has no reasons to preserve
coherence and extra care is necessary to maintain it.

Furthermore, in general we need to apply different modifications \(f\) to different
\(E_{n}\)-classes, which naturally raises concern of how to ensure that construction is performed in
a Borel way.  In applications, the modification \(f\) applied to an \(E_{n}\)-class \(C\), usually
depends on the ``shape'' of \(C\) and the \(E_{m}\)-classes it contains, but does not depend on
other \(E_{n}\)-classes. If there are only countably many such ``configurations'' of
\(E_{n}\)-classes, resulting partial actions \(f \circ \phi\) will be Borel as long as we
consistently apply the same modification whenever ``configurations'' are the same. This idea can be
formalized as follows.

\subsection{Rational sequences of partial actions}
\label{sec:rati-part-acti}

Let \((E_{n}, \phi_{n})_{n}\) be a convergent sequence of partial actions on \(X\). For an
\(E_{n}\)-class \(C\), let \(\mathcal{E}_{m}(C)\) denote the collection of \(E_{m}\)-classes
contained in \(C\).  Given two \(E_{n}\)-classes \(C\) and \(C'\), we denote by
\(\phi_{n}(C) \equiv \phi_{n}(C')\) the existence for each \(m \le n\) of a bijection
\(\mathcal{E}_{m}(C) \ni D \mapsto D' \in \mathcal{E}_{m}(C') \) such that
\(\phi_{n}(D) = \phi_{n}(D')\) for all \(D \in \mathcal{E}_{m}(C)\).  Collection of images
\(\{\phi_{n}(D) : D \in \bigcup_{m\le n} \mathcal{E}_{m}(C) \}\) constitutes the ``configuration'' of
\(C\) referred to earlier.

We say that the sequence \((E_{n}, \phi_{n})_{n}\) of partial actions is \textbf{rational} if for
each \(n\) there exists a Borel \(E_{n}\)-invariant partition \(X = \bigsqcup_{k} Y_{k}\) such that
for each \(k\) one has \(\phi_{n}(C) \equiv \phi_{n}(C')\) for all \(E_{n}\)-classes
\(C,C' \subseteq Y_{k}\).

\begin{remark}
  \label{rem:rational-partial-actions}
  This concept of rationality applies verbatim to convergent sequences of partial actions
  \((X_{n},E_{n}, \phi_{n})_{n}\) as described in Section~\ref{sec:part-acti-part}. One can check
  that such a sequence is rational if and only if the sequence \((\hat{E}_{n},\hat{\phi}_{n})\) is
  rational.
\end{remark}

\subsection{Generating the flow equivalence relation}
\label{sec:free-acti-flow}

As an application of the partial actions formalism, we show that any orbit equivalence relation
given by a free Borel \(\mathbb{R}\)-flow can also be generated by a free action of any non-discrete
and non-compact Polish group.  For this we need the following representation of an
\(\mathbb{R}\)-flow as a limit of partial \(\mathbb{R}\)-actions.

\begin{proposition}
  \label{prop:chain-bers-flow}
  Any free Borel \(\mathbb{R}\)-flow on \(X\) can be represented as a limit of a convergent rational
  sequence of partial \(\mathbb{R}\)-actions \((X_{n},E_{n}, \phi_{n})_{n}\) such that
  \begin{enumerate}
  \item both \(X_{n}\) and \(E_{n}\) are increasing: \(X_{0} \subseteq X_{1} \subseteq \cdots\) and
    \(E_{0} \subseteq E_{1} \subseteq \cdots\); (see Remark~\ref{rem:layered-variant-partial-actions})
  \item\label{item:finitely-many-subclasses} each \(E_{n+1}\)-class contains finitely many
    \(E_{n}\)-classes;
  \item\label{item:uncountable-classes} each \(E_{0}\)-class has cardinality of continuum;
  \item\label{item:uncountable-complement} for each \(E_{n+1}\)-class \(C\) the set
    \(C \setminus X_{n}\) has cardinality of continuum.
  \end{enumerate}
\end{proposition}

\begin{proof}
  Any \(\mathbb{R}\)-flow admits a rational\footnote{Rationality of the cross-section here means
    that the distance between any two points of \(\mathcal{C}\) is a rational number. More
    generally, rationality of a cross-section \(\mathcal{C}\) for an \(\mathbb{R}^{d}\)-action means
    \(r \in \mathbb{Q}^{d}\) whenever \(c + r = c'\) for some \(c, c' \in \mathcal{C}\).}
  \((-4,4)\)-lacunary cross-section (see~\cite[Section~2]{slutskyTimeChangeEquivalence2019}), which
  we denote by \(\mathcal{C}\).  Let \((E_{\mathcal{C}}, \phi_{\mathcal{C}})\) be the partial
  \(\mathbb{R}\)-action as defined in Section~\ref{sec:tess-locally-comp}.  If \(D\) is an
  \(E_{\mathcal{C}}\)-class, then \(\phi_{\mathcal{C}}(D)\) is an interval. For \(\epsilon > 0\),
  let \(D^{\epsilon}\) consist of those \( x \in D\) such that \(\phi_{\mathcal{C}}(x)\) is at least
  \(\epsilon\) away from the boundary points of \(\phi_{\mathcal{C}}(D)\). In other words,
  \(D^{\epsilon}\) is obtained by shrinking the class \(D\) by \(\epsilon\) from each side.

  The restriction of the orbit equivalence relation onto \(\mathcal{C}\) is hyperfinite. This fact
  is true in the much wider generality of actions of locally compact Abelian groups
  \cite{cottonAbelianGroupActions2022}.  Specifically for \(\mathbb{R}\)-flows, \(E|_{\mathcal{C}}\)
  is generated by the first return map---a Borel automorphism of \(\mathcal{C}\) that sends a point
  in \(\mathcal{C}\) to the next one according to the order of the \(\mathbb{R}\)-flow.  The first
  return map is well defined and is invertible, except for the orbits, where \(\mathcal{C}\) happens
  to have the maximal or the minimal point.  The latter part of the space evidently admits a Borel
  selector and is therefore smooth, hence won't affect hyperfiniteness of the equivalence relation.
  It remains to recall the standard fact that orbit equivalence relations of \(\mathbb{Z}\)-actions
  are hyperfinite (see, for instance, \cite[Theorem~5.1]{doughertyStructureHyperfiniteBorel1994}),
  and thus so is the restriction \(E|_{\mathcal{C}}\).

  In particular, we can represent the \(\mathbb{R}\)-flow as the limit of a convergent
  sequence of partial actions \((E'_{n}, \phi'_{n})_{n}\) as described in
  Section~\ref{sec:hyperf-tess-lcsc}.  Note that \((E'_{n}, \phi'_{n})_{n}\) is necessarily rational
  by rationality of \(\mathcal{C}\). Such a sequence satisfies
  items~\eqref{item:finitely-many-subclasses} and~\eqref{item:uncountable-classes}, but
  fails~\eqref{item:uncountable-complement}.  We fix this by shrinking equivalence classes to
  achieve proper containment.  Let \((\epsilon_{n})_{n}\) be a strictly decreasing sequence of
  positive reals such that \( 1 > \epsilon_{0}\) and \(\lim_{n}\epsilon_{n} = 0\). Put
  \(X'_{n} = \bigcup D^{\epsilon_{n}}\), where the union is taken over all
  \(E_{\mathcal{C}}\)-classes \(D\). Note that sets \(X'_{n}\) fail to cover \(X\), because the
  boundary points of any \(E_{\mathcal{C}}\)-class do not belong to any of \(X'_{n}\). Put
  \(Y = X \setminus \bigcup_{n} X'_{n}\) and let \(X_{n} = X'_{n} \cup Y\).  Clearly,
  \((X_{n})_{n}\) is an increasing sequence of Borel sets and \(\bigcup_{n} X_{n} = X\).

  Finally, set \(E_{n} = E'_{n}|_{X_{n}}\) and \(\phi_{n} : X_{n} \to \mathbb{R}\) to be
  \(\phi'_{n}|_{X_{n}}\).  The sequence \((X_{n},E_{n},\phi_{n})_{n}\) of partial
  \(\mathbb{R}\)-actions satisfies the conditions of the proposition.
\end{proof}

All non-smooth orbit equivalence relations produced by free Borel \(\mathbb{R}\)-flows are Borel
isomorphic to each other~\cite[Theorem~3]{kechrisCountableSectionsLocally1994}.
Theorem~\ref{thm:flow-generated-by-any-Polish-group} will show that this orbit equivalence relation
can also be generated by a free action of any non-compact and non-discrete Polish group.

Let \(G\) be a group. We say that a set \(A \subseteq G\) \textbf{admits infinitely many disjoint
  right translates} if there is a sequence \((g_{n})_{n}\) of elements of \(G\) such that
\(Ag_{m} \cap Ag_{n} = \emptyset\) for all \(m \ne n\).
\begin{lemma}
  \label{lem:disjoint-translates-in-Polish-group}
  Let \(G\) be a non-compact Polish group.  There exists a neighborhood of the identity \(V
  \subseteq G\) such that for any finite \(F \subseteq G\) the set \(VF\) admits infinitely many
  disjoint right translates.
\end{lemma}

\begin{proof}
  We begin with the following characterization of compactness established independently by
  Solecki~\cite[Lemma~1.2]{soleckiActionsNoncompactNonlocally2000} and
  Uspenskij~\cite{uspenskijSubgroupsMinimalTopological2008}: a Polish group \(G\) is non-compact if
  and only if there exists a neighborhood of the identity \(U \subseteq G\) such that
  \(F_{1}UF_{2} \ne G\) for all finite \(F_{1}, F_{2} \subseteq G\).  Let \(V \subseteq G\) be a
  symmetric neighborhood of the identity such that \(V^{2} \subseteq U\).  We claim that such a set
  \(V\) has the desired property.  Pick a finite \(F \subseteq G\), set \(g_{0} = e\) and choose
  \(g_{n}\) inductively as follows. Let \(F_{1} = F^{-1}\) and
  \(F_{2,n} = F\cdot\{g_{k} : k < n\}\).  The defining property of \(U\) assures existence of
  \(g_{n} \not \in F_{1}UF_{2,n}\).  Translates \((VFg_{n})_{n}\) are then pairwise disjoint, for if
  \(VFg_{m} \cap VFg_{n} \ne \emptyset\) for \(m < n\), then
  \(g_{n} \in F^{-1}V^{-1}VFg_{m} \subseteq F_{1}UF_{2,n}\), contradicting the construction.
\end{proof}

\begin{theorem}
  \label{thm:flow-generated-by-any-Polish-group}
  Let \(E\) be an orbit equivalence relation given by a free Borel \(\mathbb{R}\)-flow on \(X\).
  Any non-discrete non-compact Polish group \(G\) admits a free Borel action \(G \acts X\) such that
  \(E_{G} = E\).
\end{theorem}

\begin{proof}
  Let \((X_{n}, E_{n}, \phi_{n})_{n}\) be a convergent sequence of partial \(\mathbb{R}\)-actions as
  in Proposition~\ref{prop:chain-bers-flow} and let \(V \subseteq G\) be given by
  Lemma~\ref{lem:disjoint-translates-in-Polish-group}. Choose a countable dense \((h_{n})_{n}\) in
  \(G\) so that \(\bigcup_{n}Vh_{n} = G\).  Since the sequence of partial
  \(\mathbb{R}\)-actions is rational, one may pick for each \(n\) a Borel \(E_{n}\)-invariant
  partition \(X_{n} = \bigsqcup_{k}Y_{n,k}\) such that \(\phi_{n}(C) \equiv \phi_{n}(C')\) for all
  \(E_{n}\)-classes \(C,C' \subseteq Y_{n,k}\).  We construct a convergent sequence of partial
  \(G\)-actions \((X_{n}, E_{n}, \psi_{n})_{n}\) such that for each \(n\) and \(k\) there
  exists a finite set \(F \subseteq G\) such that \(\{h_{i} : i < n\} \subseteq F\) and
  \(\psi_{n}(C) = VF\) for all \(E_{n}\)-classes \(C \subseteq Y_{n,k}\).

  For any \(E_{0}\)-class \(C\), both \(\phi_{0}(C) \subseteq \mathbb{R}\) and \(V \subseteq G\) are
  Borel sets of the same cardinality.  We may therefore pick a Borel bijection
  \(f_{k} : \phi_{0}(C) \to V\) where \(C \subseteq Y_{0,k}\).  For the base
  of the inductive construction we set \(\psi_{0}|_{Y_{k}} = f_{k}\circ \phi_{0}\).  Suppose that
  \(\psi_{m} : X_{m} \to G\), \(m \le n\), have been constructed.

  We now construct \(\psi_{n+1}\).  Let \(C\) be an \(E_{n+1}\)-class and let
  \(D_{1}, \ldots, D_{l}\) be a complete list of \(E_{n}\)-classes contained in \(C\). By the
  inductive assumption, there are finite sets \(F_{1}, \ldots, F_{l} \subseteq G\) such that
  \(\psi_{n}(D_{i}) = VF_{i}\).  Let \(\tilde{F} = \bigcup_{i \le l}F_{i}\).  By the choice of
  \(V\), there are elements \(g_{1}, \ldots, g_{l} \in G\) such that \(V\tilde{F}g_{i}\), are
  pairwise disjoint for \(1 \le i \le l\). Pick a finite \(F \subseteq G\) large enough that
  \(\tilde{F}g_{i} \subseteq F\), \(\{h_{i} : i < n+1\} \subseteq F\), and
  \(VF \setminus \bigcup_{i\le l}V\tilde{F}g_{i}\) has cardinality of continuum (the latter can be
  achieved, for instance, by assuring that one more disjoint translate of \(V\tilde{F}\) is inside
  \(VF\)).  Note that
  \(\phi_{n+1}(C \setminus X_{n}) = \phi_{n+1}(C) \setminus \bigcup_{i\le l}\phi_{n+1}(D_{i})\) has
  cardinality of continuum by the properties guaranteed by
  Proposition~\ref{prop:chain-bers-flow}. Pick any Borel bijection
  \[f : \phi_{n+1}(C) \setminus \bigcup_{i\le l}\phi_{n+1}(D_{i}) \to VF \setminus \bigcup_{i}
    \psi_{n}(D_{i})g_{i}\]
  and define \(\psi_{n+1}\) by the conditions
  \(\psi_{n+1}|_{D_{i}} = \psi_{n}|_{D_{i}} \cdot g_{i}\) and
  \(\psi_{n+1}|_{C \setminus \bigcup_{i\le l}D_{i}} = f \circ \phi_{n+1}\).  Just as in the base
  case, the same modification \(f\) works for all classes \(E_{n+1}\)-classes \(C, C'\) such that
  \(\phi_{n+1}(C) \equiv \phi_{n+1}(C')\), which ensures Borelness of the construction.

  It is now easy to check that \((X_{n}, E_{n},\psi_{n})_{n}\) is a convergent sequence of partial
  \(G\)-actions, hence its limit is a free Borel action \(G \acts X\) such that \(E_{G} = E\).
\end{proof}

\begin{remark}
  \label{rem:groups-generating-hyperfinite-relations}
  Theorem~\ref{thm:flow-generated-by-any-Polish-group} highlights difference with actions of
  discrete groups, since a free Borel \(\mathbb{Z}\)-action that preserves a finite measure cannot
  be generated by a free Borel action of a non-amenable group (see, for
  instance,~\cite[Proposition~4.3.3]{zimmerErgodicTheorySemisimple1984} or
  \cite[Proposition~2.5(ii)]{jacksonCountableBorelEquivalence2002}).

  However, if we consider hyperfinite equivalence relations without any finite invariant measures,
  then we do have the analog for \(\mathbb{Z}\)-actions.  There exists a unique up to isomorphism
  non-smooth hyperfinite Borel equivalence relation without any finite invariant measures and it can
  be realized as an orbit equivalence relation of a free action of any infinite countable group
  \cite[Proposition~11.2]{doughertyStructureHyperfiniteBorel1994}.
\end{remark}

%%% Local Variables:
%%% mode: latex
%%% TeX-master: "../Katok-representation-theorem-for-Borel-flows"
%%% End:

\section{Lipschitz Maps}
\label{sec:lipsch-maps}

Our goal in this section is to prove Theorem~\ref{thm:grid-flow}, which shows that any free Borel
\(\mathbb{R}^{d}\)-flow is bi-Lipschitz orbit equivalent to a flow with an integer
grid. Sections~\ref{sec:linked-sets}--\ref{sec:lipschitz-shifts} build the necessary tools to
construct such an orbit equivalence.  Verification of the Lipschitz conditions stated in the lemmas
within these sections is straightforward and routine.  We therefore omit the arguments in the
interest of brevity.

Recall that a map \(f : X \to Y\) between metric spaces \((X,d_{Y})\) and \((Y,d_{Y})\) is
\textbf{\(K\)-Lipschitz} if \(d_{Y}(f(x_{1}), f(x_{2})) \le K d_{X}(x_{1}, x_{2})\) for all
\(x_{1}, x_{2} \in X\), and it is \textbf{\((K_{1},K_{2})\)-bi-Lipschitz} if \(f\) is injective,
\(K_{2}\)-Lipschitz, and \(f^{-1}\) is \(K^{-1}_{1}\)-Lipschitz, which can equivalently be stated as
\[K_{1}d_{X}(x_{1},x_{2}) \le d_{Y}(f(x_{1}), f(x_{2})) \le K_{2} d_{X}(x_{1}, x_{2}) \quad \textrm{
    for all } x_{1},x_{2} \in X.\]
The \textbf{Lipschitz constant} of a Lipschitz map \(f\) is the smallest \(K\) with respect to which
\(f\) is \(K\)-Lipschitz.

\subsection{Linked sets}
\label{sec:linked-sets}

Given two Lipschitz maps \( f : A \to A' \) and \( g : B \to B' \) that agree on the intersection
\( A \cap B \), the map \( f \cup g : A \cup B \to A' \cup B' \), in general, may not be
Lipschitz. The following condition is sufficient to ensure that \( f \cup g \) is Lipschitz with the
Lipschitz constant bounded by the maximum of the constants of \(f\) and \(g\).

\begin{definition}
  \label{def:star-condition}
  Let \( (X, d) \) be a metric space and \( A, B \subseteq X \) be its subsets. We say that \( A \)
  and \( B \) are \textbf{linked} if for all \( x \in A \) and \( y \in B \) there exists
  \( z \in A \cap B \) such that \( d(x,y) = d(x,z) + d(z, y) \).
\end{definition}

\begin{lemma}
  \label{lem:one}
  Let \( (X, d) \) be a metric space, \( f : A \to A' \), \( g : B \to B' \) be \( K \)-Lipschitz
  maps between subsets of \(X\) and suppose that \( f |_{A \cap B} = g |_{A \cap B} \). If
  \( A \) and \( B \) are linked, then \( f \cup g : A \cup B \to A' \cup B' \) is
  \( K \)-Lipschitz.
\end{lemma}

Recall that a metric space \((X,d)\) is \textbf{geodesic} if for all points \(x, y \in X\) there
exists a geodesic between them---an isometry \(\tau : \bigl[0, d(x,y)\bigr] \to X\) such that
\(\tau(0) = x\) and \(\tau\bigl(d(x,y)\bigr) = y\). For geodesic metric spaces, closed sets
\( A, B \subseteq X \) are always linked whenever the boundary of one of them is contained in the
other. The boundary of a set \(A\) will be denoted by \(\partial A\), and \(\inter A\) will stand for
the interior of \(A\).

\begin{lemma}
  \label{lem:satisfy-star}
  Suppose \( (X,d) \) is a geodesic metric space. If \( A, B \subseteq X \) are closed and satisfy
  \( \partial A \subseteq B \), then \( A \) and \( B \) are linked.
\end{lemma}

\subsection{Inductive step}
\label{sec:inductive-step}

The following lemma encompasses the inductive step in the construction of the forthcoming
Theorem~\ref{thm:grid-flow}.

\begin{lemma}
  \label{lem:inductive-step-extension}
  Let \((X,d)\) be a geodesic metric space and \(A \subseteq X\) be a closed set.  Suppose
  \((A_{i})_{i=1}^{n}\) are pairwise disjoint closed subsets of \(A\) and
  \(h_{i} : A_{i} \to A_{i}\) are \((K_{1},K_{2})\)-bi-Lipschitz maps such that
  \(h_{i}|_{\partial A_{i}}\) is the identity map for each \(1 \le i \le n\).  The map
  \(g : A \to A\) given by
  \begin{displaymath}
    g(x) =
    \begin{cases}
      h_{i}(x) & \textrm{ if \(x \in A_{i}\)},\\
      x & \textrm{otherwise}
    \end{cases}
  \end{displaymath}
  is \((K_{1},K_{2})\)-bi-Lipschitz.
\end{lemma}

\subsection{Lipschitz shifts}
\label{sec:lipschitz-shifts}

Let \( (X, \norm{\cdot}) \) be a normed space and let \(A \subseteq X\) be a closed bounded subset.
We begin with the following elementary and well-known observation regarding Lipschitz perturbations
of the identity map.

\begin{lemma}
  \label{lem:lipschitz-pertrubation}
  If \(\xi : A \to X\) is a \(K\)-Lipschitz map, \( K < 1\), then
  \( A \ni x \mapsto x + \xi(x) \in X\) is \((1-K, 1+K)\)-bi-Lipschitz.
\end{lemma}

For the rest of Section~\ref{sec:lipschitz-shifts}, we fix a vector \( v \in X \) and a real
\( K > ||v|| \).  Let the function \( f_{A, K, v} : A \to X \) be given by
\[ f_{A, K, v}(x) = x + \frac{d(x, \partial A)}{K} v, \]
where \(d(x, \partial A)\) denotes the distance from \(x\) to the boundary of \(A\).  This function
(as well as its variant to be introduced shortly) is
\((1-K^{-1}||v||, 1+K^{-1}||v||)\)-bi-Lipschitz.  To simplify the notation, we set
\(\alpha^{+} = 1 + K^{-1}||v||\) and \(\alpha^{-} = 1 - K^{-1}||v||\).

\begin{lemma}
  \label{lem:f-bi-Lipschitz-homeomorphism}
  The function \(f_{A,K,v}\) is an \((\alpha^{-}, \alpha^{+})\)-bi-Lipschitz homeomorphism onto
  \(A\).
\end{lemma}

Fix a real \(L > 0\) and let \( A^{L} = \{x \in A : d(x, \partial A) \ge L \} \) be the set of those
elements that are at least \( L \) units of distance away from the boundary of \(A\).

\begin{lemma}
  \label{lem:props-of-fl}
  \( f_{A,K,v}|_{A^{L}} = f_{A^{L},K,v} + LK^{-1} v \) and
  \( f_{A,K,v}(A^{L}) = A^{L} + LK^{-1}v \).
\end{lemma}

A truncated shift function \(h_{A,K,v,L} : A \to X\) is defined by

\begin{displaymath}
  h_{A,K,v,L}(x) =
  \begin{cases}
    f_{A,K,v}(x) & \text{for \(x \in A \setminus A^{L}\)}, \\
    x + LK^{-1} v & \text{for \(x \in A^{L}\)}.

  \end{cases}
\end{displaymath}

\begin{lemma}
  \label{lem:h-is-a-bi-Lipschitz-homeomorphism}
  The function \(h_{A,K,v,L}\) is an \((\alpha^{-}, \alpha^{+})\)-bi-Lipschitz homeomorphism onto
  \(A\).
\end{lemma}

\subsection{Lipschitz equivalence to grid flows}
\label{sec:application}

The maps \(h_{A,K,v,L}\) can be used to show that any free Borel \(\mathbb{R}^{d}\)-flow is
bi-Lipschitz equivalent to a flow admitting an integer grid.  This is the content of
Theorem~\ref{thm:grid-flow}, but first we formulate the properties of partial actions needed for the
construction. This is an adaption of the so-called unlayered toast construction announced
in~\cite{gaoForcingConstructionsCountable2022}.  The proof given
in~\cite[Appendix~A]{marksBorelCircleSquaring2017} for \(\mathbb{Z}^{d}\)-actions, transfers to
\(\mathbb{R}^{d}\)-flows.

For the rest of the paper, we fix a norm \(|| \cdot ||\) on \(\mathbb{R}^{d}\) and let
\(d(x,y) = ||x-y||\) be the corresponding metric on \(\mathbb{R}^{d}\).  Recall that
\(B_{R}(r) \subseteq \mathbb{R}^{d}\) denotes a closed ball of radius \(R\) centered at
\(r \in \mathbb{R}^{d}\).

\begin{lemma}
  \label{lem:separated-unlayered-toasts}
  Let \(K > 0\) be a positive real. Any free \(\mathbb{R}^{d}\)-flow on a standard Borel space \(X\)
  is a limit of a rational convergent sequence of partial actions \((X_{n}, E_{n}, \phi_{n})_{n}\) (see
  Section~\ref{sec:part-acti-part}) such that for each \(E_{n}\)-class \(C\)
  \begin{enumerate}
  \item \(\phi_{n}(C)\) is a closed and bounded subset of \(\mathbb{R}^{d}\) and
    \(B_{K}(0) \subseteq \phi_{n}(C)\);
  \item the set of \(E_{m}\)-classes, \(m \le n\), contained in \(C\) is finite;
  \item\label{item:far-from-boundary} \(d(\phi_{n}(D), \partial \phi_{n}(C)) \ge K\) for any
    \(E_{m}\)-class \(D \) such that \(D \subseteq C\).
  \end{enumerate}
\end{lemma}

Before outlining the proof, we need to introduce some notation. Let \(E_{1}, \ldots, E_{n}\) be
equivalence relations on \(X_{1}, \ldots, X_{n}\) respectively.  By \(E_{1} \vee \cdots \vee E_{n}\)
we mean the equivalence relation \(E\) on \(\bigcup_{i \le n} X_{i}\) generated by \(E_{i}\), i.e.,
\(x E y\) whenever there exist \(x_{1}, \ldots, x_{m}\) and for each \( 1\le i \le m\) there exists
\( 1 \le j(i) \le n\) such that \(x_{1} = x\), \(x_{m} = y\) and \(x_{i} E_{j(i)} x_{i+1}\) for all
\(1 \le i < m\).

If \(E\) is an equivalence relation on \(Y \subseteq X\) and \(K > 0\), we define the relation
\(E^{+K}\) on \(Y^{+K} = \{x \in X : \dist(x, y) \le K \textrm{ for some } y \in Y\}\) by
\(x_{1} E^{+K} x_{2}\) if and only if there are \(y_{1}, y_{2} \in Y\) such that
\(\dist(x_{1},y_{1}) \le K \), \(\dist(x_{2},y_{2}) \le K\) and \(y_{1} E y_{2}\).  Note that in
general, \(E^{+K}\) may not be an equivalence relation if two \(E\)-classes get connected after the
``fattening''. However, \(E^{+K}\) is an equivalence relation if \(\dist(C_{1},C_{2}) > 2K\) holds
for all distinct \(E\)-classes \(C_{1}, C_{2}\).

\begin{proof}[Proof of Lemma~\ref{lem:separated-unlayered-toasts}]
  One starts with a sufficiently fast-growing sequence of radii \(a_{n}\) (say,
  \(a_{n} = K1000^{n+1}\) is fast enough) and chooses using~\cite{boykinBorelBoundednessLattice2007}
  (see also~\cite[Lemma~A.2]{marksBorelCircleSquaring2017}) a sequence of Borel
  \(B_{a_{n}}(0)\)-lacunary cross-sections \(\mathcal{C}_{n} \subseteq X\) such that
  \begin{equation}
    \label{eq:7}
    \forall x \in X\ \forall \epsilon > 0\ \exists^{\infty} n \textrm{ such that }
    \dist(x,\mathcal{C}_{n})  < \epsilon a_{n},
  \end{equation}
  where \(\dist(x, \mathcal{C}_{n}) = \inf\{ \dist(x,c) : c \in \mathcal{C}_{n}\}\) and
  \(\exists^{\infty}\) stands for ``there exist infinitely many''. We may assume without loss of
  generality that cross-sections \(\mathcal{C}_{n}\) are rational in the sense that if
  \(c_{1} + r = c_{2}\) for some \(c_{1}, c_{2} \in \bigcup_{n}\mathcal{C}_{n}\) then
  \(r \in \mathbb{Q}^{d}\).  This can be achieved by moving elements of \(\mathcal{C}_{n}\) by an
  arbitrarily small amount (see \cite[Lemma~2.4]{slutskyTimeChangeEquivalence2019}) which maintains
  the property given in Eq.~\eqref{eq:7}.  Rationality of cross-sections guarantees that the
  sequence of partial actions constructed below is rational.

  One now defines \(X_{n}\) and \(E_{n}\) inductively with the base
  \(X_{0} = \mathcal{C}_{0} + B_{a_{0}/10}(0)\), and \(x E_{0} y\) if and only if there is
  \(c \in \mathcal{C}_{0}\) such that \(x, y \in c + B_{a_{0}/10}(0)\).  For the inductive step,
  begin with \(\tilde{X}_{n} = \mathcal{C}_{n} + B_{a_{n}/10}(0) \) and \(\tilde{E}_{n}\) being
  given analogously to the base case: \(x \tilde{E}_{n} y\) if and only if there is some
  \(c \in \mathcal{C}_{n}\) such that \(\dist(x,c) \le a_{n}/10\) and \(\dist(y,c) \le a_{n}/10\).
  Set \(E'_{n} = \tilde{E}_{n} \vee E^{+K}_{n-1} \vee \cdots \vee E_{0}^{+K}\) and let
  \(X'_{n} = \tilde{X}_{n} \cup \bigcup_{i=0}^{n-1}X_{i}^{+K}\) be the domain of
  \(E'_{n}\). Finally, let \(X_{n}\) be the \(E'_{n}\)-saturation of \(\tilde{X}_{n}\), i.e.,
  \(x \in X_{n}\) if and only if there exists \(y \in \tilde{X}_{n}\) such that \(x E'_{n} y\).  Put
  \(E_{n} = E'_{n}|_{X_{n}}\).

  An alternative description of an \(E_{n}\)-class is as follows. One starts with an
  \(\tilde{E}_{n}\)-class \(C_{n}\) and joins it first with all \(E^{+K}_{n-1}\)-classes \(D\) that
  intersect \(C_{n}\).  Let the resulting \(\tilde{E}_{n}\vee E^{+K}_{n-1}\)-class be denoted by
  \(C_{n-1}\).  Next we add all \(E^{+K}_{n-2}\)-classes that intersect \(C_{n-2}\) producing an
  \(\tilde{E}_{n}\vee E^{+K}_{n-1} \vee E^{+K}_{n-2}\)-class \(C_{n-2}\).  The process terminates
  with an \(E_{n}\)-class \(C_{0}\).

  It is easy to check inductively that the diameter of any \(E_{n}\)-class \(C\) satisfies
  \(\diam(C) \le a_{n}/3\) and therefore \(\dist(C_{1}, C_{2}) \ge a_{n}/3 \gg 2K\) for all distinct
  \(E_{n}\)-classes \(C_{1},C_{2}\) by the lacunarity of \(\mathcal{C}_{n}\).  The latter shows that
  \(E^{+K}_{n}\) is an equivalence relation on \(X_{n}^{+K}\).

  Monotonicity of the sequence \((X_{n},E_{n})_{n}\) is evident from the construction.
  Eq.~\eqref{eq:7} is crucial for establishing the fact that \(\bigcup_{n} X_{n} = X\).  Indeed, for
  each \(x \in X\) there exists some \(n\) such that \(\dist(x, \mathcal{C}_{n}) < a_{n}/10\) and
  thus also \(x \in \tilde{X}_{n} \subseteq X_{n}\).

  The maps \(\phi_{n} : X_{n} \to \mathbb{R}^{d}\), needed to specify partial
  \(\mathbb{R}^{d}\)-actions, are defined by the condition \(\phi_{n}(x)c = x\) for the unique
  \(c \in \mathcal{C}_{n}\) such that \(c E_{n} x\).  Note that
  \(d(\phi_{n}(D), \partial \phi_{n}(C)) \ge K\) for any \(E_{m}\)-class \(D\), \(m < n\), that is
  contained in an \(E_{n}\)-class \(C\) is a consequence of the fact that \(D^{+K} \subseteq C \) by
  the construction.  The convergent sequence of partial actions \((X_{n},E_{n}, \phi_{n})_{n}\)
  therefore satisfies the desired properties.
\end{proof}

Let \(\mathfrak{F}_{1}\) and \(\mathfrak{F}_{2}\) be free \(\mathbb{R}^{d}\)-flows on \(X\) that
generate the same orbit equivalence relation, \(E_{\mathfrak{F}_{1}} = E_{\mathfrak{F}_{2}}\), and
let
\(\rho = \rho_{\mathfrak{F}_{1}, \mathfrak{F}_{2}} : \mathbb{R}^{d} \times X \to \mathbb{R}^{d}\) be
the associated cocycle map, defined for \(x \in X\) and \(r \in \mathbb{R}^{d}\) by the condition
\(x +_{2} r = x +_{1} \rho(r,x)\). We say that the cocycle \(\rho\) is
\((K_{1},K_{2})\)-bi-Lipschitz if such is the map
\(\rho(\,\cdot\, , x) : \mathbb{R}^{d} \to \mathbb{R}^{d}\) for all \(x \in X\):
\begin{equation}
  \label{eq:2}
  K_{1}||r_{2} -r_{1}|| \le ||\rho(r_{2},x) - \rho(r_{1},x)|| \le K_{2} ||r_{2}-r_{1}||.
\end{equation}
Since
\(\rho(r_{2},x) - \rho(r_{1},x) = \rho(r_{2} - r_{1},x +_{1} r_{1})\),
Lipschitz condition~\eqref{eq:2} for a cocycle can be equivalently and more concisely stated as
\begin{equation}
  \label{eq:3}
K_{1} \le \dfrac{||\rho(r,x)||}{||r||} \le K_{2} \quad \textrm{ for all \(x \in X\) and \(r \in
  \mathbb{R}^{d} \setminus \{0\}\)}.
\end{equation}

\begin{remark}
  \label{rem:cocyle-inverse}
  Note that cocycles \(\rho_{\mathfrak{F}_{1}, \mathfrak{F}_{2}}\) and
  \(\rho_{\mathfrak{F}_{2}, \mathfrak{F}_{1}}\) are connected via the identities
\[\rho_{\mathfrak{F}_{1}, \mathfrak{F}_{2}} (\rho_{\mathfrak{F}_{2}, \mathfrak{F}_{1}}(r,x),x) = r
  \qquad \textrm{ and } \qquad \rho_{\mathfrak{F}_{2}, \mathfrak{F}_{1}} (\rho_{\mathfrak{F}_{1},
    \mathfrak{F}_{2}}(r,x),x) = r.\]
In particular, if \(\rho_{\mathfrak{F}_{1}, \mathfrak{F}_{2}}\) is \((K_{1},K_{2})\)-bi-Lipschitz,
then \(\rho_{\mathfrak{F}_{2}, \mathfrak{F}_{1}}\) is \((K_{2}^{-1}, K_{1}^{-1})\)-bi-Lipschitz.
\end{remark}

\begin{definition}
  \label{def:integer-grid}
  Let \(\mathfrak{F}\) be a free \(\mathbb{R}^{d}\)-flow on \(X\). An \textbf{integer grid} for the
  flow \(\mathfrak{F}\) is a \(\mathbb{Z}^{d}\)-invariant Borel subset \(Z \subseteq X\) whose
  intersection with each orbit of the flow is a \(\mathbb{Z}^{d}\)-orbit.  In other words,
  \(Z + \mathbb{R}^{d} = X\), \(Z + \mathbb{Z}^{d} = Z\), and
  \(z_{1}+ \mathbb{Z}^{d} = z_{2} + \mathbb{Z}^{d}\) for all \(z_{1}, z_{2} \in Z\) such that
  \(z_{1}E_{\mathfrak{F}}z_{2}\).
\end{definition}

Not every flow admits an integer grid, but, as the following theorem shows, each flow is bi-Lipschitz
equivalent to the one that does.

\begin{theorem}
  \label{thm:grid-flow}
  Let \(\mathfrak{F}_{1}\) be a free Borel \(\mathbb{R}^{d}\)-flow on \(X\). For any \(\alpha > 1\)
  there exists a free Borel \(\mathbb{R}^{d}\)-flow \(\mathfrak{F}_{2}\) on \(X\) that admits an integer
  grid, induces the sames orbit equivalence as does \(\mathfrak{F}_{1}\), i.e.,
  \(E_{\mathfrak{F}_{1}} = E_{\mathfrak{F}_{2}}\), and whose associated cocycle
  \(\rho_{\mathfrak{F}_{1}, \mathfrak{F}_{2}}\) is \((\alpha^{-1}, \alpha)\)-bi-Lipschitz.
\end{theorem}

\begin{proof}
  Let \(R\) be so big that the ball \(B_{R}(0) \subseteq \mathbb{R}^{d}\) satisfies
  \(\mathbb{Z}^{d} + B_{R}(0) = \mathbb{R}^{d}\).  Choose \(K > 0\) large enough to ensure that
  \(\alpha^{-} = 1 - K^{-1}R > \alpha^{-1}\), and therefore also
  \(\alpha^{+} = 1 + K^{-1}R < \alpha\).  Let \((X_{n}, E_{n}, \phi_{n})_{n}\) be a rational
  convergent sequence of partial actions produced by Lemma~\ref{lem:separated-unlayered-toasts} for
  the chosen value of \(K\).  For an \(E_{n}\)-class \(C\), let \(C'\) denote the collection of all
  \(x \in C\) that are at least \(K\)-distance away from the boundary of \(C\):
  \[C' = \{x \in C : d(\phi_{n}(x), \partial \phi_{n}(C)) \ge K\}.\]
  If \(D\) is an \(E_{m}\)-class such that \(D \subseteq C\), then
  item~\eqref{item:far-from-boundary} of Lemma~\ref{lem:separated-unlayered-toasts} guarantees the
  inclusion \(D \subseteq C'\). Let \(X'_{n} = \bigcup C'\), where the union is taken over all
  \(E_{n}\)-classes \(C\), and set \(E_{n}' = E_{n}|_{X'_{n}}\), \(\phi'_{n}= \phi_{n}|_{X'_{n}}\).
  Note that \((X'_{n}, E'_{n}, \phi'_{n})_{n}\) is a rational convergent sequence of partial actions
  whose limit is the flow \(\mathfrak{F}_{1}\).  The flow \(\mathfrak{F}_{2}\) will be constructed
  as the limit of partial actions \((X'_{n},E'_{n},\psi_{n})\), where maps \(\psi_{n}\) will be
  defined inductively and will satisfy \(\psi_{n}(C') = \phi_{n}(C')\) for all \(E_{n}\)-classes
  \(C\). The sets \(Z_{n} = \psi_{n}^{-1}(\mathbb{Z}^{d})\) will satisfy
  \(Z_{m}\cap X'_{n} \subseteq Z_{n}\) for \(m \le n\), and \(Z = \bigcup_{n}Z_{n}\) will be an
  integer grid for \(\mathfrak{F}_{2}\).

    \begin{figure}[htb]
    \centering
    \begin{tikzpicture}[scale=1.8]
      \draw[rounded corners] (0,0) rectangle (6,4);
      \draw[dashed, rounded corners] (0.4,0.4) rectangle (5.6,3.6);
      \draw[<->] (2,3.6) -- (2,4) node[anchor=west,pos=0.5] {\(K\)};
      \draw (5.8,3.8) node {\(C\)};
      \foreach \x in {0.5, 0.7, ..., 5.6} {
        \foreach \y in {0.5, 0.7, ..., 3.6} {
          \filldraw (\x, \y) circle (0.1pt);
        }
      }
      \begin{scope}
        \draw (1.0,1.0) rectangle (2.4,2.8);
        \foreach \x in {1.4, 1.6, ..., 2.1} {
          \foreach \y in {1.4, 1.6, ..., 2.5} {
            \draw (\x,\y) node[cross=1.5pt,rotate=10] {};
          }
        }
      \end{scope}
      \begin{scope}[xshift=2.2cm,yshift=-0.2cm]
        \draw (1.0,1.0) rectangle (2.4,2.8);
        \foreach \x in {1.4, 1.6, ..., 2.1} {
          \foreach \y in {1.4, 1.6, ..., 2.5} {
            \draw (\x,\y) node[cross=1.5pt,rotate=10] {};
          }
        }
      \end{scope}
      \filldraw[white] (5.4, 3.4) rectangle (5.59, 3.59);
      \filldraw[white] (2.2, 2.6) rectangle (2.39, 2.79);
      \filldraw[white] (4.4, 2.4) rectangle (4.59, 2.59);
      \draw (5.45,3.45) node {\(C'\)};
      \draw (2.25,2.65) node {\(D_{1}\)};
      \draw (4.45,2.45) node {\(D_{2}\)};
    \end{tikzpicture}
    \caption{Construction of the integer grid}
    \label{fig:integer-grid-construction}
  \end{figure}
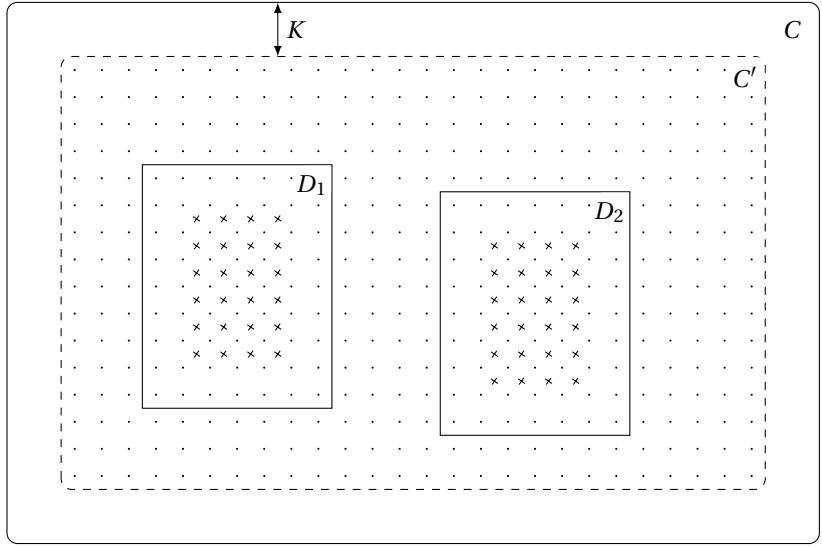

  For the base of the construction, set \(\psi_{0} = \phi'_{0}\) and
  \(Z_{0} = \psi_{0}^{-1}(\mathbb{Z}^{d})\).  Next, consider a typical \(E_{1}\)-class \(C\) with
  \(D_{1}, \ldots, D_{l}\) being a complete list of \(E_{0}\)-classes contained in it (see
  Figure~\ref{fig:integer-grid-construction}).  Consider the set
  \(\tilde{Z}_{C'} = \phi_{1}^{-1}(\mathbb{Z}^{d}) \cap C'\), which is the integer grid inside
  \(C'\) (marked by dots in Figure~\ref{fig:integer-grid-construction}).  Each of the
  \(D_{i}\)-classes comes with the grid
  \(\tilde{Z}_{D'_{i}} = \psi_{0}^{-1}(\mathbb{Z}^{d})\cap D'_{i}\) constructed at the previous
  stage (depicted by crosses in Figure~\ref{fig:integer-grid-construction}).  The coherence
  condition for partial actions guarantees existence of some \(s_{i} \in \mathbb{R}^{d}\),
  \(i \le l\), such that
  \[\phi_{1}(D'_{i})= \phi_{0}(D'_{i}) + s_{i} = \psi_{0}(D'_{i}) + s_{i}.\]

  In general, the grid \(\tilde{Z}_{C'}\) does not contain \(\tilde{Z}_{D'_{i}} \), but for
  each \(i \le l\), we can find a vector \(v_{i} \in \mathbb{R}^{d}\) of norm \(||v_{i}||\le R\)
  such that \(\tilde{Z}_{D'_{i}} +_{1} v_{i} \subseteq \tilde{Z}_{C'}\).  More precisely, we
  take for \(v_{i}\) any vector in \(B_{R}(0)\) such that \(s_{i} + v_{i} \in \mathbb{Z}^{d}\),
  which exists by the choice of \(R\).  Let \(h_{i} : \phi_{1}(D_{i}) \to \phi_{1}(D_{i})\) be the
  function \(h_{\phi_{1}(D_{i}), K, v_{i}, K}\), which is \((\alpha^{-}, \alpha^{+})\)-bi-Lipschitz
  by Lemma~\ref{lem:h-is-a-bi-Lipschitz-homeomorphism}.  Finally, define
  \(g_{1} : \phi_{1}(C') \to \phi_{1}(C')\) to be
  \begin{displaymath}
    g_{1}(r) =
    \begin{cases}
      h_{i}(r) & \textrm{if \(r \in \phi_{1}(D_{i})\)},\\
      r & \textrm{otherwise}.
    \end{cases}
  \end{displaymath}
  Lemma~\ref{lem:inductive-step-extension} has been tailored specifically to show that \(g_{1}\) is
  \((\alpha^{-}, \alpha^{+})\)-bi-Lipschitz.  We set
  \(\psi_{1}|_{C'} = g_{1} \circ \phi_{1}|_{C'}\).   Note that
  \begin{equation}
    \label{eq:6}
    \begin{split}
      \psi_{1}(D'_{i}) &= g_{1} \circ \phi_{1}(D'_{i}) = h_{i} \circ \phi_{1}(D'_{i}) =
                         \phi_{1}(D'_{i}) + KK^{-1}v_{i} \\
                       &= \phi_{0}(D'_{i}) + s_{i} + v_{i} = \psi_{0}(D'_{i}) + s_{i} + v_{i},
    \end{split}
  \end{equation}
  which validates coherence and, in view of \(s_{i} + v_{i} \in \mathbb{Z}^{d}\), gives
  \(\psi^{-1}_{1}(\mathbb{Z}^{d}) \cap D'_{i} = \psi_{0}^{-1}(\mathbb{Z}^{d}) \cap D'_{i} \) for all
  \(i \le l\).

  While we have provided the definition of \(\psi_{1}\) on a single \(E_{1}\)-class \(C\), the same
  construction can be done in a Borel way across all \(E_{1}\)-classes \(C\) using rationality of
  the sequence of partial actions just like we did in
  Theorem~\ref{thm:flow-generated-by-any-Polish-group}. If we let
  \(Z_{1} = \psi_{1}^{-1}(\mathbb{Z}^{d})\), then \(Z_{0} \cap X_{1} \subseteq Z_{1}\) by Eq.~\eqref{eq:6}.

  The general inductive step is analogous.  Suppose that we have constructed maps \(\psi_{k}\) for
  \(k \le n\).  An \(E_{n+1}\)-class \(C\) contains finitely many subclasses
  \(D_{1}, \ldots, D_{l}\), where \(D_{i}\) is an \(E_{m_{i}}\)-class, \(m_{i} < n\), and no \(D_{i}\) is contained
  in a bigger \(E_{m}\)-class for some \(m_{i} < m < n\).  By coherence and inductive assumption,
  there exist \(s_{i} \in \mathbb{R}^{d}\), \(i \le l\), such that
  \[\phi_{n+1}(D'_{i}) = \phi_{m_{i}}(D'_{i}) + s_{i} = \psi_{m_{i}}(D'_{i}) + s_{i}. \]
  Choose vectors \(v_{i} \in B_{R}(0)\) to satisfy \(s_{i} + v_{i} \in \mathbb{Z}^{d}\), set
  \(h_{i} : \phi_{n+1}(D_{i}) \to \phi_{n+1}(D_{i})\) to be \(h_{\phi_{n+1}(D_{i}), K, v_{i}, K}\),
  and define an \((\alpha^{-}, \alpha^{+})\)-bi-Lipschitz function \(g_{n+1}\) by
  \begin{displaymath}
          g_{n+1}(r) =
    \begin{cases}
      h_{i}(r) & \textrm{if \(r \in \phi_{n+1}(D_{i})\)},\\
      r & \textrm{otherwise}.
    \end{cases}
  \end{displaymath}
  Finally, set \(\psi_{n+1}|_{C'} = g_{n+1} \circ \phi_{n+1}|_{C'} \) and extend this definition to
  a Borel map \(\psi_{n+1} : X'_{n+1} \to \mathbb{R}^{d}\) using the rationality of the sequence of
  partial actions.  Coherence of the maps \((\psi_{k})_{k \le n+1}\) and the inclusion
  \(Z_{m}\cap X'_{n+1} \subseteq Z_{n+1}\) for \(m \le n+1\) follow from the analog of
  Eq.~\eqref{eq:6}.

  It remains to check the bi-Lipschitz condition for the resulting cocycle
  \(\rho_{\mathfrak{F}_{1}, \mathfrak{F}_{2}}\).  It is easier to work with the cocycle
  \(\rho_{\mathfrak{F}_{2}, \mathfrak{F}_{1}}\), which for \(x, x+r \in X'_{n}\) satisfies
  \[\rho_{\mathfrak{F}_{2}, \mathfrak{F}_{1}}(r,x) = g_{n}(\phi_{n}(x) + r) - g_{n}(\phi_{n}(x)), \]
  and is therefore \((\alpha^{-}, \alpha^{+})\)-bi-Lipschitz, because so is \(g_{n}\). Hence,
  \(\rho_{\mathfrak{F}_{2}, \mathfrak{F}_{1}}\)is also \((\alpha^{-1}, \alpha)\)-bi-Lipschitz,
  because \(\alpha^{-1} < \alpha^{-} < \alpha^{+} < \alpha\) by the
  choice of \(K\).  Finally, we apply Remark~\ref{rem:cocyle-inverse} to conclude that
  \(\rho_{\mathfrak{F}_{1}, \mathfrak{F}_{2}}\) is also \((\alpha^{-1}, \alpha)\)-bi-Lipschitz.
\end{proof}

Restricting the action of \(\mathfrak{F}_{2}\) onto the integer grid \(Z\), we get the following
corollary.

\begin{corollary}
  \label{cor:existence-of-bi-Lipschitz-grid-aciton}
  Let \(\mathfrak{F}\) be a free Borel \(\mathbb{R}^{d}\)-flow on \(X\). For any \(\alpha > 1\)
  there exist a cross-section \(Z \subseteq X\) and a free \(\mathbb{Z}^{d}\)-action \(T\) on \(Z\)
  such that the cocycle \(\rho = \rho_{\mathfrak{F}, T}: \mathbb{Z}^{d} \times X \to \mathbb{R}^{d}\)
  given by \(T_{n}x = x + \rho(n,x)\) is \((\alpha^{-1}, \alpha)\)-bi-Lipschitz.
\end{corollary}

%%% Local Variables:
%%% mode: latex
%%% TeX-master: "../Katok-representation-theorem-for-Borel-flows"
%%% End:

\section{Special representation theorem}
\label{sec:spec-repr-theor}

The main goal of this section is to formulate and prove a Borel version of Katok's special
representation theorem \cite{katokSpecialRepresentationTheorem1977} that connects free
\(\mathbb{R}^{d}\)-flows with free \(\mathbb{Z}^{d}\)-actions. We have already done most of the work
in proving Theorem~\ref{thm:grid-flow}, and it is now a matter of defining special representations
in the Borel context and connecting them to our earlier setup.

\subsection{Cocycles}
\label{sec:cocycles}

Given a Borel action \(G \acts X\), a (Borel) \textbf{cocycle} with values in a group \(H\) is a
(Borel) map \(\rho : G \times X \to H\) that satisfies the \textbf{cocycle identity}:
\[\rho(g_{2}g_{1},x) = \rho(g_{2},g_{1}x) \rho(g_{1},x) \quad \textrm{for all \(g_{1},g_{2} \in G \) and
    \(x \in X\)}.\]
We are primarily concerned with the Abelian groups \(\mathbb{Z}^{d}\) and \(\mathbb{R}^{d}\) in this
section, so the cocycle identity will be written additively.  A cocycle \(\rho : G \times X \to H\)
is said to be \textbf{injective} if \(\rho(g,x) \ne e_{H}\) for all \(g \ne e_{G}\) and all
\(x \in X\), where \(e_{G}\) and \(e_{H}\) are the identity elements of the corresponding groups.
Suppose that furthermore the groups \(G\) and \(H\) are locally compact.  We say that \(\rho\)
\textbf{escapes to infinity} if for all \(x \in X\), \(\lim_{g \to \infty}\rho(g,x) = +\infty\) in
the sense that for any compact \(K_{H} \subseteq H\) there exists a compact \(K_{G} \subseteq G\)
such that \(\rho(g,x) \not \in K_{H}\) whenever \(g \not \in K_{G}\).

\begin{example}
  \label{exmpl:cocycle-from-action}
  Suppose \(a_{H} : H \acts X\) and \(a_{G} : G \acts Y\), \(Y \subseteq X\), are free actions of
  groups \(G\) and \(H\) on standard Borel spaces, and suppose that we have containment of orbit
  equivalence relations \(E_{G} \subseteq E_{H}\). For each \(y \in Y\) and \(g \in G\), there
  exists a unique \(\rho_{a_{H},a_{G}}(g,y) \in H\) such that
  \(a_{H}(\rho_{a_{H},a_{G}}(g,y),y) = a_{G}(g,y) \).  The map
  \((g,y) \mapsto \rho_{a_{H},a_{G}}(g,y)\) is an injective Borel cocycle. We have already
  encountered two instances of this idea in Section~\ref{sec:application}.
\end{example}

\subsection{Flow under a function}
\label{sec:flow-under-function}

Borel \( \mathbb{R} \)-flows and \( \mathbb{Z} \)-actions are tightly connected through the ``flow
under a function'' construction. Let \( T : Z \to Z \) be a free Borel automorphism of a standard
Borel space and \( f : Z \to \mathbb{R}^{>0} \) be a positive Borel function. There is a natural
definition of a flow \( \mathfrak{F} : \mathbb{R} \acts X \) on the space
\( X = \{ (z, t) : z \in Z, 0 \le t < f(z)\} \) under the graph of \( f \). The action
\( (z, t) + r \) for a positive \( r \) is defined by shifting the point \( (z,t) \) by \( r \)
  units upward until the graph of \( f \) is reached, then jumping to the point \((Tz, 0)\), and
  continuing to flow upward until the graph of \(f\) at \(Tz\) is reached, etc. More formally,
\[ (z, t) + r = \bigl(T^{k}z, t + r - \sum\limits_{i=0}^{k-1}f(T^{i}z)\bigr) \]
for the unique \( k \ge 0 \) such that
\( \sum_{i=0}^{k-1}f(T^{i}z) \le t + r < \sum_{i=0}^{k}f(T^{i}z) \); for \( r \le 0 \) the action is
defined by ``flowing backward'', i.e.,
\[ (z, t) + r = \bigl(T^{-k}z, t + r + \sum\limits_{i=1}^{k}f(T^{-i}z)\bigr) \]
for \( k \ge 0 \) such that \( 0 \le t + r + \sum\limits_{i=1}^{k}f(T^{-i}z) < f(T^{-k}z) \).  The
action is well-defined provided that the fibers within the orbits of \( T \) have infinite
cumulative lengths:
\begin{equation}
  \label{eq:suspension-flow-condition}
  \sum_{i=0}^{\infty}f(T^{i}z) = +\infty\quad \textrm{ and }
  \quad \sum_{i=0}^{\infty}f(T^{-i}z) = +\infty\quad \textrm{ for all } z \in Z.
\end{equation}

\medskip

The appealing geometric picture of the ``flow under a function'' does not generalize to higher
dimensions, but admits an interpretation as the so-called special flow construction suggested
in~\cite{katokSpecialRepresentationTheorem1977}.

\subsection{Special flows}
\label{sec:special-flows}

Let \(T\) be a free \(\mathbb{Z}^{d}\)-action on a standard Borel space \(Z\) and let
\(\rho : \mathbb{Z}^{d} \times Z \to \mathbb{R}^{d}\) be a Borel cocycle.  One can construct a
\(\mathbb{Z}^{d}\)-action \(\hat{T}\), the so-called principal \(\mathbb{R}^{d}\)-extension, defined
on \(Z \times \mathbb{R}^{d}\) via \(\hat{T}_{n}(z, r) = (T_{n}z, r + \rho(n,z))\). An easy
application of the cocycle identity verifies axioms of the action.  While the action \(T\) will
typically have complicated dynamics, the action \(\hat{T}\) admits a Borel transversal as long as
the cocycle \(\rho\) escapes to infinity.

\begin{lemma}
  \label{lem:sufficient-condition-smoothness}
  If the cocycle \(\rho\) satisfies \(\lim_{n \to \infty}||\rho(n,z)|| = +\infty\) for all
  \(z \in Z\), then the action \(\hat{T}\) has a Borel transversal.
\end{lemma}

\begin{proof}
  Let \(Y_{k} = \{ (z, r) \in Z \times \mathbb{R}^{d} : ||r|| \le k\}\).  We claim that each orbit
  of \(\hat{T}\) intersects \(Y_{k}\) in a finite (possibly empty) set.  Indeed, cocycle values
  escaping to infinity yield for any \((z,r) \in Z \times \mathbb{R}^{d}\) a number \(N\) so
  large that \(||r + \rho(n,z)|| > k\) whenever \(||n|| \ge N\).  In particular, \(||n|| \ge N\)
  implies \(\hat{T}_{n}(z,r) = (T_{n}z, r + \rho(n,z)) \not \in Y_{k}\).  Hence, the intersection of
  the orbit of \((z,r)\) with \(Y_{k}\) is finite.

  Set
  \(Y = \bigsqcup_{k \in \mathbb{N}} \bigl(Y_{k} \setminus \bigcup_{n \in \mathbb{Z}^{d}}
  \hat{T}_{n}Y_{k-1}\bigr)\).  Each orbit of \(\hat{T}\) intersects \(Y\) in a finite and
  necessarily non-empty set, so \(E_{\hat{T}}|_{Y}\) is a finite Borel equivalence relation. A Borel
  transversal for \(E_{\hat{T}}|_{Y}\) is also a transversal for the action of \(\hat{T}\).
\end{proof}

We assume now that the cocycle \(\rho\) satisfies the assumptions of
Lemma~\ref{lem:sufficient-condition-smoothness}, and \(X = (Z \times \mathbb{R}^{d})/E_{\hat{T}}\)
therefore carries the structure of a standard Borel space as a push-forward of the factor map
\(\pi : Z \times \mathbb{R}^{d} \to X\), which sends a point to its \(E_{\hat{T}}\)-equivalence
class.

There is a natural \(\mathbb{R}^{d}\)-flow \(\hat{\mathfrak{F}}\) on \(Z \times \mathbb{R}^{d}\)
which acts by shifting the second coordinate: \((z,r) +_{\hat{\mathfrak{F}}} s = (z, r + s)\). This
flow commutes with the \(\mathbb{Z}^{d}\)-action \(\hat{T}\) and therefore projects onto the flow
\(\mathfrak{F}\) on \(X\) given by the condition
\(\pi((z,r) +_{\hat{\mathfrak{F}}} s) = \pi(z,r) +_{\mathfrak{F}} s\). We say that \(\mathfrak{F}\)
is the \textbf{special flow} over \(T\) generated by the cocycle \(\rho\).  Freeness of \(T\)
implies freeness of \(\mathfrak{F}\).

The construction outlined above, works just as well in the context of ergodic theory, where the
space \(Z\) would be endowed with a finite measure \(\nu\) preserved by the action \(T\).  The
product of \(\nu\) with the Lebesgue measure on \(\mathbb{R}^{d}\) induces a measure \(\mu\) on
\(X\), which is preserved by the flow \(\mathfrak{F}\).  Furthermore, \(\mu\) is finite provided the
cocycle \(\rho\) is integrable in the sense of \cite[Condition (J),
p.~122]{katokSpecialRepresentationTheorem1977}.  Katok's special representation theorem asserts
that, up to a null set, any free ergodic measure-preserving flow can be obtained via this
process. Furthermore, the cocycle can be picked to be bi-Lipschitz with Lipschitz constants
arbitrarily close to \(1\).

As will be shown shortly, such a representation result continues to hold in the framework of
descriptive set theory, and every free Borel \(\mathbb{R}^{d}\)-flow is Borel isomorphic to a
special flow over some free Borel \(\mathbb{Z}^{d}\)-action.  Moreover, just as in Katok's original
work, Theorem~\ref{thm:katoks-special-representation-flows} provides some significant control on the
cocycle that generates the flow, tightly coupling the dynamics of the \(\mathbb{Z}^{d}\)-action with
the dynamics of the flow it produces.  But first, we re-interpret the construction in different
terms.

\subsection{Flows generated by admissible cocycles}
\label{sec:flows-gener-admiss}

Let the map \( Z \ni z \mapsto (z,0) \in Z \times \{0\}\) be denoted by \(\iota\). If the cocycle
\(\rho\) is injective, then \(\pi \circ \iota : Z \to \pi(Z\times\{0\}) = Y\) is a bijection and
\(Y\) intersects every orbit of \(\mathfrak{F}\) in a non-empty countable set.  The
\(\mathbb{Z}^{d}\)-action \(T\) on \(Z\) can be transferred via \(\pi \circ \iota\) to give a free
\(\mathbb{Z}^{d}\)-action \(T' = \pi \circ \iota \circ T \circ \iota^{-1} \circ \pi^{-1}\) on \(Y\).
Let \(\rho' = \rho_{T', \mathfrak{F}}: \mathbb{Z}^{d} \times Y \to \mathbb{R}^{d}\) be the cocycle
of the action \(\pi \circ \iota \circ T \circ \iota^{-1} \circ \pi^{-1}\); in other words
\begin{equation}
  \label{eq:1}
  T'_{n}(y) = (\pi \circ
\iota \circ T_{n} \circ \iota^{-1} \circ \pi^{-1})(y) = y +_{\mathfrak{F}} \rho'(n,y) \quad \textrm{
for all \(n \in \mathbb{Z}^{d}\) and \(y \in Y\)}.
\end{equation}
If \(y =  (\pi \circ \iota)(z)\) for \(z \in Z\), then Eq.~\eqref{eq:1} translates into
\[\pi(T_{n}z, 0) = \pi(z, \rho'(n,y)).\]
Since \(\pi(T_{n}z, 0) = \pi(z, \rho(-n, T_{n}z)) = \pi(z, -\rho(n,z))\), we conclude that
\(\rho'(n,y) = - \rho(n,z)\), where \(y = (\pi \circ \iota) (z)\). In particular, \(Y\) is a
discrete cross-section for the flow \(\mathfrak{F}\) precisely because \(\rho\) escapes to infinity.

Conversely, if \(\mathfrak{F}\) is any free \(\mathbb{R}^{d}\)-flow on a standard Borel space \(X\),
and \(Z \subseteq X\) is a discrete cross-section with a \(\mathbb{Z}^{d}\)-action \(T\) on it, then
\(\mathfrak{F}\) is isomorphic to the special flow over \(T\) generated by the (necessarily
injective) cocycle \(-\rho_{T,\mathfrak{F}}\).

Let us say that a cocycle \(\rho\) is \textbf{admissible} if it is both injective and escapes to
infinity. The discussion of the above two paragraphs can be summarized by saying that, up to a
change of sign in the cocycles, representing a flow as a special flow generated by an
admissible cocycle is the same thing as finding a free \(\mathbb{Z}^{d}\)-action on a discrete
cross-section of the flow.

For instance, given any free \(\mathbb{Z}^{d}\)-action \(T\) on \(Z\), we may consider the admissible
cocycle \(\rho(n, z) = -n\) for all \(z \in Z\) and \(n \in \mathbb{Z}^{d}\). The set
\(Y = \pi(Z \times \{0\})\) is then an integer grid for the flow \(\mathfrak{F}\) (in the sense of
Definition~\ref{def:integer-grid}).  Conversely, any flow that admits an integer grid is isomorphic
to a special flow generated by such a cocycle.

\subsection{Special representation theorem}
\label{sec:spec-repr-theorem}

Restriction of the orbit equivalence relation of any \(\mathbb{R}^{d}\)-flow onto a cross-section
gives a hyperfinite equivalence relation \cite[Theorem~1.16]{jacksonCountableBorelEquivalence2002},
and therefore can be realized as an orbit equivalence relation by a free Borel
\(\mathbb{Z}^{d}\)-action (as long as the restricted equivalence relation is aperiodic).  Since any
free flow admits a discrete (in fact, lacunary) aperiodic cross-section, it is isomorphic to a
special flow over \emph{some} action generated by \emph{some} cocycle. In general, however, the
structure of the \(\mathbb{Z}^{d}\)-orbit and the corresponding orbit of the flow have little to do
with each other. Theorem~\ref{thm:grid-flow} and
Corollary~\ref{cor:existence-of-bi-Lipschitz-grid-aciton} allow us to improve on this and find a
special representation generated by a bi-Lipschitz cocycle.

For comparison, Katok's theorem \cite{katokSpecialRepresentationTheorem1977} can be formulated in
the parlance of this section as follows.

\begin{theorem-nn}[Katok]
  \label{thm:katok-special-representation-flows-ergodic-theory}
  Pick some \(\alpha > 1\).  Any free ergodic measure-preserving \(\mathbb{R}^{d}\)-flow on a
  standard Lebesgue space is isomorphic to a special flow over a free ergodic measure-preserving
  \(\mathbb{Z}^{d}\)-action generated by an \((\alpha^{-1},\alpha)\)-bi-Lipschitz cocycle.
\end{theorem-nn}

As is the case with all ergodic theoretical results, isomorphism is understood to hold up to a set
of measure zero.  We may now conclude with a Borel version of Katok's special representation
theorem, which holds for all free Borel \(\mathbb{R}^{d}\)-flows and establishes isomorphism on all
orbits.

\begin{theorem}
  \label{thm:katoks-special-representation-flows}
  Pick some \(\alpha > 1\).  Any free Borel \(\mathbb{R}^{d}\)-flow is isomorphic to a special flow
  over a free Borel \(\mathbb{Z}^{d}\)-action generated by an \((\alpha^{-1},\alpha)\)-bi-Lipschitz
  cocycle.
\end{theorem}

\begin{proof}
  Let \(\mathfrak{F}\) be a free Borel \(\mathbb{R}^{d}\)-flow on
  \(X\). Corollary~\ref{cor:existence-of-bi-Lipschitz-grid-aciton} gives a cross-section
  \(Z \subseteq X\) and a \(\mathbb{Z}^{d}\)-action \(T\) on it such that the cocycle
  \(\rho_{\mathfrak{F},T} : \mathbb{Z}^{d} \times X \to \mathbb{R}^{d}\) is
  \((\alpha^{-1}, \alpha)\)-bi-Lipschitz.  By the discussion in
  Section~\ref{sec:flows-gener-admiss}, this gives a representation of the flow as a special flow
  over \(T\) generated by the cocycle \(-\rho_{\mathfrak{F},T}\), which is also
  \((\alpha^{-1},\alpha)\)-bi-Lipschitz.
\end{proof}

%%% Local Variables:
%%% mode: latex
%%% TeX-master: "../Katok-representation-theorem-for-Borel-flows"
%%% End:

\bibliography{refs.bib}
\bibliographystyle{amsplain}

\end{document}